\documentclass[11pt]{amsart}
\usepackage{amsmath,amssymb,amsfonts,amsthm,amscd,indentfirst,MnSymbol}
\usepackage{amsmath,amssymb,amsfonts,amsthm,amscd}
\usepackage{indentfirst}
\usepackage{hyperref}

\numberwithin{equation}{section}

\newtheorem{prop}{Proposition}[section]

\newtheorem{lemm}[prop]{Lemma}

\def\and{\quad{\rm and}\quad}

\def\<{\langle}
\def\>{\rangle}

\usepackage{graphicx}
\newtheorem{theorem}{Theorem}

\newtheorem{question}{Question}
\newtheorem{remark}{Remark}
\newtheorem*{theoremA}{Theorem A}
\newtheorem*{theoremB}{Theorem B}
\newtheorem*{theoremC}{Theorem C}

\title[Monge-Amp\`ere equation with bounded periodic data]{Monge-Amp\`ere equation with bounded periodic data}
\author{ Yanyan Li
        \and
        Siyuan Lu 
        }       
\address{Department of Mathematics, Rutgers University, 110 Frelinghuysen Road, Piscataway, NJ 08854}
\email{yyli@math.rutgers.edu}
\email{siyuan.lu@math.rutgers.edu}
\thanks{Research of the first named author is partially supported by NSF grant DMS-1501004.}

\begin{document}

\begin{abstract}
We consider the Monge-Amp\`ere equation $\det(D^2u)=f$ in $\mathbb{R}^n$, where $f$ is a positive bounded periodic function. We prove that $u$ must be the sum of a quadratic polynomial and a periodic function. For $f\equiv 1$, this is the classic result by J\"orgens, Calabi and Pogorelov. For $f\in C^\alpha$, this was proved by Caffarelli and the first named author.
\end{abstract}


\maketitle

\section{Introduction}

A classic theorem of J\"orgens \cite{J}, Calabi \cite{Calabi} and Pogorelov \cite{P} states that any classical convex solution of
\begin{align*}
\det(D^2u)=1\quad in \quad \mathbb{R}^n
\end{align*}
must be a quadratic polynomial.

A simpler and more analytical proof, along the lines of affine geometry, was later given by Cheng and Yau \cite{CY}. The theorem was extended by Caffarelli \cite{C} to viscosity solutions. Another proof of the theorem was given by Jost and Xin \cite{JX}. Trudinger and Wang \cite{TW} proved that if $\Omega$ is an open convex subset of $\mathbb{R}^n$ and $u$ is a convex $C^2$ solution of $\det(D^2u)=1$ in $\Omega$ with $\lim_{x\rightarrow \partial \Omega}u(x)=\infty$, then $\Omega=\mathbb{R}^n$.   Ferrer, Mart\'inez and Mil\'an \cite{FMM1,FMM2} extended the above Liouville type theorem in dimension two. 
Caffarelli and the first named author \cite{CL, CL2} made two extensions, and one of them includes periodic data.

More specificly,  assume for some $a_1,\cdots,a_n>0$, $f$ satisfies
\begin{align}\label{periodic}
f(x+a_ie_i)=f(x),\quad  \forall x\in \mathbb{R}^n,\quad 1\leq i\leq n,
\end{align}
where $e_1=(1,0,\cdots,0),\cdots, e_n=(0,\cdots,0,1)$.

Consider the Monge-Amp\`ere equation
\begin{align}\label{eqn}
\det(D^2u)=f,\quad in \quad \mathbb{R}^n.
\end{align}
\begin{theoremA}\ (\cite{CL2})\ 
Let $f\in C^\alpha(\mathbb{R}^n)$, $0<\alpha<1$ with $f>0$ satisfy (\ref{periodic}), and let $u\in C^2(\mathbb{R}^n)$ be a convex solution of (\ref{eqn}). Then there exist $b\in \mathbb{R}^n$ and a symmetric positive definite $n\times n$ matrix $A$ with $\det A=\strokedint_{\prod_{1\leq i\leq n}[0,a_i]}f$, such that $v:=u-\frac{1}{2}x^T Ax-b\cdot x$ is $a_i$-periodic in $i$-th variable, i.e.
\begin{align*}
v(x+a_ie_i)=v(x), \quad \forall x\in \mathbb{R}^n,\quad 1\leq i\leq n.
\end{align*}
\end{theoremA}

\medskip

For applications, it is desirable to study the problem with less regularity assumption on $f$. It was conjectured in \cite{CL2}, see Remark 0.5 there, that Theorem A remains valid for $f\in L^\infty(\mathbb{R}^n)$ satisfying $0<\inf_{\mathbb{R}^n} f\leq \sup_{\mathbb{R}^n}f<\infty$. We confirm the conjecture in Theorem \ref{main theorem} below.

\medskip

We first recall the definition of a solution of (\ref{eqn}) in the Alexandrov sense.

Let $u$ be a convex function in an open set $\Omega$ of $\mathbb{R}^n$. For $y\in \Omega$, denote
\begin{align*}
\nabla u(y)=\{p\in \mathbb{R}^n|u(x)\geq u(y)+p\cdot(x-y),\forall x\in \Omega\}
\end{align*}
the generalized gradient of $u$ at $y$.

For $f\in L^\infty(\Omega)$ with $f\geq 0$ a.e.,  $u$ is called a solution of
\begin{align*}
\det(D^2u)=f,\quad in \quad \Omega
\end{align*}
in the Alexandrov sense if $u$ is a convex function in $\Omega$ and $|\nabla u(O)|=\int_Of$, for every open set $O\subset \Omega$.

Similarly, for a symmetric $n\times n$ matrix $A$, we say that $v\in C^{0,1}(\Omega)$ is a solution
\begin{align*}
\det(A+D^2v)=f,\quad in \quad \Omega
\end{align*}
in the Alexandrov sense if $u:=\frac{1}{2}x^TAx+v$ is convex in $\Omega$ and satisfies
\begin{align*}
\det(D^2u)=f,\quad in \quad \Omega
\end{align*}
in the Alexandrov sense.

\medskip

Our first result is the existence and uniqueness of periodic solutions for $f\in L^\infty$.
\begin{theorem}\label{existence}
Let $f\in L^\infty(\mathbb{R}^n)$ with $0<\inf_{\mathbb{R}^n} f\leq \sup_{\mathbb{R}^n} f<\infty$ satisfy (\ref{periodic}) a.e., and let $A$ be a symmetric positive definite  $n\times n$ matrix satisfying
\begin{align*}
\det A=\strokedint_{\prod_{1\leq i\leq n}[0,a_i]}f.
\end{align*}
Then there exists a unique (up to addition of constants) $v\in C^{0,1}(\mathbb{R}^n)$ which is $a_i$-periodic in the $i$-th variable, such that
\begin{align}\label{eq-existence}
\det(A+D^2v)=f,\quad in\quad \mathbb{R}^n
\end{align}
in the Alexandrov sense. Moreover, $v\in C^{1,\alpha}(\mathbb{R}^n)$ for some $0<\alpha<1$.
\end{theorem}

\begin{remark}
If $f\geq 0$, the existence part still holds by passing to limit.
\end{remark}

\begin{remark} 
If the smoothness assumption of $f$ in Theorem \ref{existence} is strengthened to $f\in C^{k,\alpha}(\mathbb{R}^n)$, $k\geq 0$, $0<\alpha<1$, there exists a solution $v\in C^{k+2,\alpha}(\mathbb{R}^n)$. For $k\geq 4$, the method in \cite{Li} is applicable; for $0\leq k\leq 3$, this can be established by a smooth approximation of $f$ based on the $C^{2,\alpha}$ theory of Caffarelli in \cite{C3}, together with the $C^0$ estimate of solutions in \cite{Li}. A different proof of  these results under the assumption that $0<f\in C^{k,\alpha}(\mathbb{R}^n)$, $k\geq 0$, $0<\alpha<1$, was given in \cite{C4}. Monge-Amp\`ere equations on Hessian manifolds were studied in \cite{CY2} and \cite{CV}.
 \end{remark}

\medskip

Now we state our main theorem.

\begin{theorem}\label{main theorem}
Let $f\in L^\infty(\mathbb{R}^n)$ with $0<\inf_{\mathbb{R}^n} f\leq \sup_{\mathbb{R}^n}f<\infty$ satisfy (\ref{periodic}) a.e., 
and let $u$ be a solution of (\ref{eqn}) in the Alexandrov sense. Then there exist $b\in \mathbb{R}^n$ and a symmetric positive definite $n\times n$ matrix $A$ with $\det A=\strokedint_{\prod_{1\leq i\leq n}[0,a_i]}f$, such that $v:=u-\frac{1}{2}x^T Ax-b\cdot x$ is $a_i$-periodic in the $i$-th variable. Moreover,  $v\in C^{1,\alpha}(\mathbb{R}^n)$ for some $0<\alpha<1$.
\end{theorem}

\begin{question}
Does the conclusion of the theorem, except for the 
$C^{1,\alpha}$ regularity of $v$,
 still hold if $f\geq 0$?
\end{question}

The main difficulty in proving Theorem \ref{main theorem} is that $C^2$ estimates on $u$ are no longer valid since $f$ is only bounded, which can be seen from the counter examples in \cite{Wang}. The proof in \cite{CL2} for Theorem A makes use of the fact that $D^2u$ is uniformly bounded in a non-trivial way, thus we can not carry out the same proof in the current setting. The key observation in our proof is that we can still prove the main propositions in \cite{CL2} without the uniform bounds of $D^2u$, which also enables us to simplify the proof of Thereom A in several ways. The proof of Theorem \ref{main theorem} follows closely the main steps in \cite{CL2}.

\medskip

The organization of the paper is as follows: in Section 2, we state two theorems on linearized Monge-Amp\`ere equations established by Caffarelli and Guti\'errez \cite{CG} which play crucial roles in the proof of Theorem \ref{main theorem}. In Section 3, we prove Theorem \ref{existence} about the existence and uniqueness of solutions on $\mathbb{T}^n$ which is used in the proof of Theorem \ref{main theorem}. In Section 4, we give the proof of Theorem \ref{main theorem}. We will mainly focus on the part that is different from \cite{CL2}.

\section{Preliminary}
In this section, we state two theorems on linearized Monge-Amp\`ere equations.

\begin{theoremB}\label{harnark}  \ (\cite{CG}) \
Let $\Omega$ be an open convex subset of $\mathbb{R}^n$ satisfying $B_1\subset \Omega \subset B_n$, $n\geq 2$, and let $\phi\in C^2(\bar{\Omega})$ be a convex function satisfying, for some cosntants $\lambda$ and $\Lambda$,
\begin{align*}
\begin{cases}
0<\lambda\leq \det(D^2\phi)\leq \Lambda<\infty,\quad in\quad \Omega,\\
\phi=0,\quad on\quad \partial \Omega.
\end{cases}
\end{align*}
Let $a_{ij}=\det(D^2\phi)\phi^{ij}$ be the linearization of the Monge-Amp\`ere operator at $u$.

(1) Assume that $v\in C^2(\Omega)$ satisfies
\begin{align*}
a_{ij}v_{ij}\geq f,\quad v\geq 0,\quad in \quad \Omega.
\end{align*}
Then for any $p>0, r>s>0$, there exists some $C(n,\lambda,\Lambda,p,r,s)>0$, such that
\begin{align*}
\sup_{x\in \Omega, dist(x,\partial\Omega)>r}v\leq C\left(\|v\|_{L^p(x\in \Omega, dist(x,\partial\Omega)>s)}+\|f\|_{L^n(x\in \Omega, dist(x,\partial\Omega)>s)}\right).
\end{align*}

(2) Assume that $v\in C^2(\Omega)$ satisfies
\begin{align*}
a_{ij}v_{ij}\leq f,\quad v\geq 0,\quad in \quad \Omega.
\end{align*}
Then for $r>s>0$, there exist $p_0>0$ and $C(n,\lambda,\Lambda,p_0,r,s)>0$, such that
\begin{align*}
\|v\|_{L^{p_0}(x\in \Omega, dist(x,\partial\Omega)>s)}\leq C\left(\inf_{x\in \Omega, dist(x,\partial\Omega)>r}v+\|f\|_{L^n(x\in \Omega, dist(x,\partial\Omega)>s)}\right).
\end{align*}
\end{theoremB}
\begin{proof}
We notice that Theorem 1 and Theorem 4 in \cite{CG} hold for supersolutions, and thus the measure part of the proof of Lemma 4.1 in \cite{CG} holds for subsolutions, the rest follows exactly those of Theorem 4.8 in \cite{CC}. We remark that (1) is called local maximum principle and (2) is called weak Harnack inequality in literature.
\end{proof}

\begin{theoremC}\label{harnark 2}   \ (\cite{CG})\
Let $\Omega$ and $\tilde{\Omega}$ be open convex subsets of $\mathbb{R}^n$ satisfying $B_1\subset \Omega,\tilde{\Omega} \subset B_n$, $n\geq 2$, and let $\phi\in C^2(\bar{\Omega})$ and $\tilde{\phi}\in C^2(\bar{\tilde{\Omega}})$ be convex functions satisfying, for some constants $\lambda$ and $\Lambda$,
\begin{align*}
\begin{cases}
0<\lambda\leq \det(D^2\phi)\leq \Lambda<\infty,\quad in\quad \Omega,\\
\phi=0,\quad on\quad \partial \Omega.
\end{cases}
\end{align*}
\begin{align*}
\begin{cases}
0<\lambda\leq \det(D^2\tilde{\phi})\leq \Lambda<\infty,\quad in\quad \tilde{\Omega},\\
\tilde{\phi}=0,\quad on\quad \partial \tilde{\Omega}.
\end{cases}
\end{align*}
Let $a_{ij}=\det(D^2\phi)\phi^{ij}$ and $\tilde{a}_{ij}=\det(D^2\tilde{\phi})\tilde{\phi}^{ij}$ be the linearizations of the Monge-Amp\`ere operator at $u$ and $\tilde{u}$ respectively.

Assume that $v\in C^2(\Omega)$ with $v\geq 0$ satisfies
\begin{align*}
\begin{cases}
a_{ij}v_{ij}\geq 0,\quad in \quad \Omega,\\
 \tilde{a}_{ij}v_{ij}\leq 0,\quad in \quad \tilde{\Omega}.
\end{cases}
\end{align*}

Let $O\subset \bar{O}\subset \Omega\cap\tilde{\Omega}$ be an open set, then there exist constants $\alpha(n,\lambda,\Lambda,O)$ and $C(n,\lambda,\Lambda,O)$ such that
\begin{align*}
\sup_{O}v\leq C\inf_{O}v,
\end{align*}
and
\begin{align*}
\|v\|_{C^\alpha(O)}\leq C.
\end{align*}
\end{theoremC}

\section{Proof of Theorem \ref{existence}}
We now prove Theorem \ref{existence}. This is based on the result in \cite{Li}, together with the regularity theory of Caffarelli \cite{C2}.
\begin{proof}
Since Monge-Amp\`ere equations are affine invariant, we may assume without loss of generality that $a_i=1$ for all $i$, and $f$ satisfies $\int_{[0,1]^n}f=1$.
For convenience, we identify peroidic functions as functions on $\mathbb{T}^n$.

We first establish the existence part.

Let $\rho\in C_c^\infty(B_1)$, $\int_{B_1}\rho=1$. For $\epsilon>0$, $\rho_\epsilon(x)=\epsilon^{-n}\rho(\epsilon x)$, let
\begin{align}\label{mollification}
f_\epsilon(x)=\int_{\mathbb{R}^n}\rho_\epsilon(x-y)f(y)dy
\end{align}
be the mollification of $f$. It is clear that $f_\epsilon$ is periodic. Define 
\begin{align*}
\tilde{f}_\epsilon=f_\epsilon-\strokedint_{\mathbb{T}^n}f_\epsilon+\det A.
\end{align*} 
It follows that $\strokedint_{\mathbb{T}^n}\tilde{f}_\epsilon=\det A$. 
By Theorem 2.2 in \cite{Li}, there exists a unique function $\tilde{v}_\epsilon\in C^\infty(\mathbb{T}^n)$ with $(A+D^2\tilde{v}_\epsilon)>0$, $\int_{\mathbb{T}^n}\tilde{v}_\epsilon=0$ satisfying
\begin{align*}
\det(A + D^2\tilde{v}_\epsilon)=\tilde{f}_\epsilon,\quad on\quad \mathbb{T}^n
\end{align*}
and $|\tilde{v}_\epsilon|+|\nabla \tilde{v}_\epsilon|\leq C(A)$ on $\mathbb{T}^n$. Passing to a subsequence, $\tilde{v}_\epsilon\rightarrow v$ in $ C^0(\mathbb{T}^n)$ and $v$ is a solution of (\ref{eq-existence}) in the Alexandrov sense, see e.g. Proposition 2.6 in \cite{Fi}.
The $C^{1,\alpha}$ regularity of $v$ for some $\alpha\in (0, 1)$ follows 
from Theorem 2 in \cite{C2}.

\medskip

Now we establish the uniqueness part. Suppose that there exist two solutions $v$ and $\hat{v}$. Without loss of generality, assume
\begin{align*}
\min_{\mathbb{T}^n}(v-\hat{v})=0.
\end{align*}

Then $u(x):=\frac{1}{2}x^TAx+v(x)$ and $\hat{u}(x):=\frac{1}{2}x^TAx+\hat{v}(x)$ are solutions of (\ref{eqn}) in the Alexandrov sense. 

Since $v$ is bounded, we can find $M>0$ large enough such that 
\begin{align*}
\Omega_M=\{x\in \mathbb{R}^n|u(x)<M\}
\end{align*}
constains $[-2,2]^n$.

Let $u_{\epsilon}\in C^0(\bar{\Omega}_M)\cap C^\infty(\Omega_M)$ be the solution of the following Dirichlet problem (see e.g. Proposition 2.4 in \cite{CL})
\begin{align*}
\begin{cases}
\det(D^2u_{\epsilon}(x))=f_\epsilon(x),\quad in \quad \Omega_M,\\
u_{\epsilon}(x)=M,\quad on\quad \partial\Omega_M.
\end{cases}
\end{align*}
By a barrier argument, $u_\epsilon(x)-M\geq -Cdist(x,\partial\Omega_M)^{\frac{2}{n}}$ if $n=3$ and $u_\epsilon(x)-M\geq -Cdist(x,\partial\Omega_M)^{\alpha}$ for some $0<\alpha<1$ if $n=2$, see e.g. \cite{C1} or Lemma A.1 in \cite{CL}. Since $f_\epsilon\rightarrow f$ in $L^1(\Omega_M)$ as $\epsilon\rightarrow 0$, it follows that $u_{\epsilon}\rightarrow \tilde{u}$ in $C^0(\bar{\Omega}_M)$ along a subsequence as $\epsilon\rightarrow 0$. As mentioned earlier $\tilde{u}$ satisfies $\det(D^2\tilde{u})=f$ in the Alexandrov sense. By the uniqueness of solution to Dirichlet problem in the Alexandrov sense, e.g. Corollary 2.11 in \cite{Fi}, we have $\tilde{u}=u$.

Similarly there exists a convex solution $\hat{u}_\epsilon\in  C^\infty([-2,2]^n)$ satisfying
\begin{align*}
\det(D^2\hat{u}_\epsilon(x))=f_\epsilon(x),\quad in \quad [-2,2]^n,
\end{align*}
with $\hat{u}_\epsilon \rightarrow\hat{u}$ in $C^0([-2,2]^n)$.

For any function $w$, denote $F(D^2w)=\det^{\frac{1}{n}}(D^2w)$ and $F_{ij}(D^2w)=\frac{\partial F}{\partial w_{ij}}$, since $F$ is concave, we have
\begin{align*}
F(D^2\hat{u}_{\epsilon})\leq F(D^2u_{\epsilon})+F_{ij}(D^2u_{\epsilon})\partial_{ij}(\hat{u}_{\epsilon }-u_{\epsilon }),
\end{align*}
i.e.
\begin{align*}
F_{ij}(D^2u_{\epsilon})\partial_{ij}(u_{\epsilon }-\hat{u}_{\epsilon })\leq 0.
\end{align*}
Similarly,
\begin{align*}
F(D^2u_{\epsilon})\leq F(D^2\hat{u}_{\epsilon})+F_{ij}(D^2\hat{u}_{\epsilon})\partial_{ij}(u_{\epsilon }-\hat{u}_{\epsilon }),
\end{align*}
i.e.
\begin{align*}
F_{ij}(D^2\hat{u}_{\epsilon})\partial_{ij}(u_{\epsilon }-\hat{u}_{\epsilon})\geq 0.
\end{align*}

Let $\delta_\epsilon=\min_{[-2,2]^n}( u_\epsilon-\hat{u}_\epsilon )$, then
\begin{align*}
u_\epsilon-\hat{u}_\epsilon-\delta_\epsilon\geq 0,\quad on\quad [-2,2]^n.
\end{align*}

Now by Theorem B and Theorem C, we have
\begin{align*}
\max_{[-1,1]^n} ( u_\epsilon-\hat{u}_\epsilon-\delta_\epsilon )\leq C\min_{[-1,1]^n}( u_{\epsilon}-\hat{u}_{\epsilon}-\delta_\epsilon ).
\end{align*}
Let $\epsilon\rightarrow 0$,  we have $\lim_{\epsilon\rightarrow 0}\delta_\epsilon=0$ on $\mathbb{T}^n$ as $\min_{\mathbb{T}^n}(v-\hat{v})=0$. It follows that
\begin{align*}
\max_{[-1,1]^n}(u-\hat{u})\leq C\min_{[-1,1]^n}(u-\hat{u})=0.
\end{align*}
Thus $u=\hat{u}$ on $[-1,1]^n$. It follows that $v=\hat{v}$ on $\mathbb{T}^n$. The theorem is now proved.
\end{proof}

\section{Proof of Theorem \ref{main theorem}}
We now start to prove Theorem \ref{main theorem}. As mentioned in the introduction, we will follow the main steps in \cite{CL2}. 

By the affine invariance of the problem, we may assume without loss of generality that $a_i=1$ for all $i$, and $f$ satisfies $\int_{[0,1]^n}f=1$.

We first note that Proposition 2.1 and its proof in \cite{CL2} still hold in the current setting.
\begin{prop}\label{asymp bound}
There exist a symmetric positive definite $n\times n$ matrix $A$ with $\det A=1$ and postive constants $\delta$ and $C_1$, such that
\begin{align*}
|u(x)-\frac{1}{2}x^T Ax|\leq C_1|x|^{2-\delta},\quad \forall |x|\geq 1.
\end{align*}
\end{prop}

\medskip

For nonzero $e\in \mathbb{R}^n$, as in \cite{CL2}, we define the second incremental quotient,
\begin{align*}
\Delta^2_eu(x)=\frac{u(x+e)+u(x-e)-2u(x)}{\|e\|^2}
\end{align*}
where $\|e\|$ denotes the Euclidean norm of $e$.

Let 
\begin{align*}
E=\{k_1e_1+\cdots+k_ne_n;k_1,\cdots,k_n \textit{ are integers, } k_1^2+\cdots+k_n^2>0\}.
\end{align*}

The following is analogous to Lemma 2.4 in \cite{CL2}. 
\begin{prop}\label{second incre bound}
\begin{align*}
\gamma:=\sup_{e\in E}\sup_{y\in \mathbb{R}^n}\Delta^2_eu(y)<\infty.
\end{align*}
\end{prop}

\begin{proof}
We will follow the main steps as in \cite{CL2}, with some modifications. 

For any $M>0$, define
\begin{align*}
\Omega_M=\{x\in\mathbb{R}^n|u(x)<M\}.
\end{align*}
By John's lemma, there exists an affine transformation 
\begin{align*}
A_M=a_Mx+b_M
\end{align*}
such that
\begin{align*}
B_R\subset A_M(\Omega_M)\subset B_{nR}
\end{align*}
with $\det a_M=1$. Denote
\begin{align*}
O_M=\frac{1}{R}a_M(\Omega_M).
\end{align*}
Define
\begin{align*}
u_M(x)=\frac{1}{R^2}u(a_M^{-1}(Rx)),\quad x\in O_M.
\end{align*}

Now for $e\in E$ and $y\in \mathbb{R}^n$, let $x=\frac{1}{R}a_M(y)$. Take $M$ large so that $y\in \Omega_{\frac{M}{2}}$. It follows from Propostion \ref{asymp bound} that 
\begin{align*}
dist(x,\partial O_M)\geq \frac{1}{C_0}
\end{align*}
where $C_0$ depends only on $n, \inf f$ and $\sup f$.

Let  $\tilde{e}=\frac{1}{R}a_M(e)$, then
\begin{align*}
\Delta_e^2u(y)&=\frac{u(y+e)+u(y-e)-2u(y)}{\|e\|^2}\\
&=\frac{u(a^{-1}_M(R(x+\tilde{e})))+u(a^{-1}_M(R(x-\tilde{e})))-2u(a^{-1}_M(Rx))}{\|e\|^2}\\
&= \frac{R^2\|\tilde{e}\|^2}{\|e\|^2}\Delta^2_{\tilde{e}}u_M(x)\\
&=\frac{\|a_M(e)\|^2}{\|e\|^2}\Delta_{\tilde{e}}^2u_M(x).
\end{align*}

In the rest of the proof, we use $C$ to denote various positive constants depending only on $n, \inf f, \sup f$ and the constants $\delta$ and $C_1$ in Proposition \ref{asymp bound}.

By Proposition \ref{asymp bound}, $C^{-1}\leq \frac{M}{R^2}\leq C$, $\|a_M\|\leq C$. The proposition will follow as long as $\Delta_{\tilde{e}}^2u_M(x)\leq C$ for $dist(x,\partial O_M)\geq \frac{1}{C_0}$.

We now prove $\Delta_{\tilde{e}}^2u_M(x)\leq C$.

Note that $u_M(x)$ satisfies
\begin{align*}
\begin{cases}
\det(D^2u_M(x))=f(a_M^{-1}(Rx)),\quad in \quad O_M,\\
u_M(x)=\frac{M}{R^2},\quad on\quad \partial O_M.
\end{cases}
\end{align*}

Let $f_\epsilon$ be the mollification of $f$ given by (\ref{mollification}) and let $u_{M,\epsilon}(x)$ be the solution of the following Dirichlet problem
\begin{align*}
\begin{cases}
\det(D^2u_{M,\epsilon}(x))=f_\epsilon(a_M^{-1}(Rx)),\quad in \quad O_M,\\
u_{M,\epsilon}(x)=\frac{M}{R^2},\quad on\quad \partial O_M.
\end{cases}
\end{align*}

As in the proof of Theorem \ref{existence}, we have $u_{M,\epsilon}\rightarrow u_M$ in $C^0(\bar{O}_M)$ as $\epsilon\rightarrow0$.

By Lemma 2.2 in \cite{CL2}, $u_{M,\epsilon}$ satisfies
\begin{align}\label{subsolution}
F_{ij}(D^2u_{M,\epsilon}(x))\partial_{ij}\left(\Delta^2_{\tilde{e}}u_{M,\epsilon}(x)\right)\geq 0,\quad x\in\Omega_M,\quad dist(x,\partial O_M)\geq \frac{1}{8C_0}.
\end{align}

By Lemma A.1 in \cite{CL2}, we have
\begin{align*}
\int_{x\in O_M,dist(x,\partial O_M)\geq \frac{1}{4C_0}}\Delta^2_{\tilde{e}}u_{M,\epsilon}(x)\leq C.
\end{align*}

Together with Theorem B, we have
\begin{align*}
\Delta^2_{\tilde{e}}u_{M,\epsilon}(x)\leq C
\end{align*}
for $x\in O_M$ with $dist(x,\partial O_M)\geq \frac{1}{2C_0}$.

Let $\epsilon\rightarrow 0$, we have
\begin{align*}
\Delta^2_{\tilde{e}}u_{M}(x)\leq C
\end{align*}
for $x\in O_M$ with $dist(x,\partial O_M)\geq \frac{1}{C_0}$.

The proposition is now proved.

\end{proof}

\medskip

For $\lambda\geq 1$ and any function $v$, let
\begin{align*}
v^\lambda(x)=\frac{v(\lambda x)}{\lambda^2},\quad x\in \mathbb{R}^n.
\end{align*}

Denote 
\begin{align*}
Q(x)=\frac{1}{2}x^TAx.
\end{align*}

\begin{lemm}\label{c1mu}
There exists a constant $\mu\in (0,1)$ such that,
\begin{align*}
u^\lambda\rightarrow Q\quad in \quad C^{1,\mu}_{loc}(\mathbb{R}^n)\quad as \quad \lambda\rightarrow\infty.
\end{align*}
\end{lemm}

\begin{proof}
By Proposition \ref{asymp bound}, $u^\lambda\rightarrow Q$ in $C^0_{loc}(\mathbb{R}^n)$ as $\lambda\rightarrow\infty$. For $r>0$, denote $D_r=\{x\in \mathbb{R}^n|Q(x)<r^2\}$. There exists $\lambda_1>0$ such that for $\lambda\geq \lambda_1$, 
\begin{align*}
D_{\frac{3}{2}}\subset \{u^\lambda<4\}=:\Omega_{4,\lambda}\subset D_{\frac{5}{2}}
\end{align*}

We know that $\det(D^2u^\lambda)=f(\lambda x)$ in $\Omega_{4,\lambda}$ in the Alexandrov sense. By Theorem 2 in \cite{C2}, there exist $\mu^\prime\in (0,1)$ and $C\geq 1$ depending only on $n,\inf f,\sup f$ and $A$ such that
\begin{align*}
\|u^\lambda\|_{C^{1,\mu^\prime}(D_{\frac{4}{3}})}\leq C.
\end{align*}

Thus we have $u^\lambda\rightarrow Q$ in $C^{1,\mu}(D_1)$ for $0<\mu<\mu^\prime<1$. The lemma follows given the fact that $u^\lambda(x)=a^2u^{\lambda a}(\frac{x}{a})$ for all $a,\lambda>0$ and $x\in \mathbb{R}^n$.

\end{proof}

The following proposition is Proposition 2.3 in \cite{CL2}.

\begin{prop}\label{concavity}
\begin{align*}
\sup_{\mathbb{R}^n}\Delta^2_eu=\frac{e^\prime A e}{\|e\|^2},\quad \forall e\in E.
\end{align*}
\end{prop}

\begin{proof}
Denote
\begin{align*}
\alpha=\sup_{\mathbb{R}^n}\Delta^2_eu,\quad \beta=\frac{e^\prime A e}{\|e\|^2}.
\end{align*}

For $\lambda>0$, $\hat{e}=\frac{e}{\lambda}$, by strict convexity (see e.g. \cite{C1}) and Proposition \ref{second incre bound}, we have
\begin{align*}
0<\Delta^2_{\hat{e}}u^\lambda(x)=\Delta_e^2u(\lambda x)\leq \alpha<\infty,\quad x\in \mathbb{R}^n.
\end{align*}

By Lemma A.2 in \cite{CL2} and Lemma \ref{c1mu}, we have
\begin{align*}
\lim_{\lambda\rightarrow \infty}\int_{B_1}\Delta^2_{\hat{e}}u^\lambda dx=\int_{B_1}\beta dx=\beta|B_1|.
\end{align*}

Thus $\alpha\geq \beta$. Now suppose $\alpha>\beta$, let $\beta<\beta^\prime<\alpha^\prime<\alpha^{\prime\prime}<\alpha$, we have

\begin{align*}
\limsup_{\lambda\rightarrow\infty}\left(\alpha^\prime|\{\Delta^2_{\hat{e}}u^\lambda\geq \alpha^\prime\}\cap B_1|\right)\leq \lim_{\lambda\rightarrow \infty}\int_{B_1}\Delta^2_{\hat{e}}u^\lambda dx=\beta|B_1|.
\end{align*}

Thus for all large $\lambda$, we have 
\begin{align*}
\alpha^\prime |\{\Delta^2_{\hat{e}}u^\lambda\geq \alpha^\prime\}\cap B_1|\leq \beta^\prime|B_1|.
\end{align*}
i.e.
\begin{align*}
\frac{|\{\Delta^2_{\hat{e}}u^\lambda\leq \alpha^\prime\}\cap B_1|}{|B_1|}\geq \frac{\alpha^\prime-\beta^\prime}{\alpha^\prime}.
\end{align*}

For $M>0$, denote 
\begin{align*}
\Omega_{M,\lambda}=\{x\in \mathbb{R}^n|u^\lambda(x)<M\}.
\end{align*}
By Lemma \ref{c1mu}, there exist $M,\lambda_1$ such that for $\lambda>\lambda_1$, we have $B_2\subset \Omega_{M,\lambda}$.

As in the proof of Proposition \ref{second incre bound}, let $f_\epsilon$ be the mollification of $f$ given by (\ref{mollification}), let $u^\lambda_{M,\epsilon}(x)$ be the solution of the following Dirichlet problem
\begin{align*}
\begin{cases}
\det(D^2u^\lambda_{M,\epsilon}(x))=f_\epsilon(\lambda x),\quad in \quad \Omega_{M,\lambda},\\
u^\lambda_{M,\epsilon}(x)=M,\quad on\quad \partial\Omega_{M,\lambda}.
\end{cases}
\end{align*}

Then we have $u^\lambda_{M,\epsilon}\rightarrow u^\lambda$ in $C^0(\bar{\Omega}_{M,\lambda})$ as $\epsilon\rightarrow 0$, see in the proof of Theorem \ref{existence}.

For $\epsilon$ small enough, we have
\begin{align*}
\frac{|\{\Delta^2_{\hat{e}}u^\lambda_{M,\epsilon}\leq \alpha^{\prime\prime}\}\cap B_1|}{|B_1|}\geq \frac{\alpha^\prime-\beta^\prime}{\alpha^\prime}.
\end{align*}

By (\ref{subsolution}), $\Delta^2_{\hat{e}}u^\lambda_{M,\epsilon}$ is a subsolution of the linearized Monge-Amp\`ere equation at $u^\lambda_{M,\epsilon}$.

Apply Theorem B, we have, for some $p_0>0$ and $C>0$,
\begin{align*}
\|\alpha-\Delta^2_{\hat{e}}u^\lambda_{M,\epsilon}\|_{L^{p_0}(B_1\cap\{\Delta^2_{\hat{e}}u^\lambda_{M,\epsilon}\leq \alpha^{\prime\prime}\} )}\leq \|\alpha-\Delta^2_{\hat{e}}u^\lambda_{M,\epsilon}\|_{L^{p_0}(B_1)}\leq C\inf_{B_{\frac{3}{4}}}\left(\alpha-\Delta^2_{\hat{e}}u^\lambda_{M,\epsilon}\right).
\end{align*}

Consequently,
\begin{align*}
(\alpha-\alpha^{\prime\prime})|B_1\cap\{\Delta^2_{\hat{e}}u^\lambda_{M,\epsilon}\leq \alpha^{\prime\prime}\}|^{\frac{1}{p_0}}\leq C\inf_{B_{\frac{3}{4}}}\left(\alpha-\Delta^2_{\hat{e}}u^\lambda_{M,\epsilon}\right).
\end{align*}

Therefore,
\begin{align*}
\sup_{B_{\frac{3}{4}}}\Delta^2_{\hat{e}}u^\lambda_{M,\epsilon}\leq \alpha-C^{-1}
\end{align*}
for all $\lambda>\lambda_1$. 

Let $\epsilon\rightarrow0$, then
\begin{align*}
\sup_{B_{\frac{\lambda}{2}}}\Delta^2_{e}u=\sup_{B_{\frac{1}{2}}}\Delta^2_{\hat{e}}u^\lambda\leq \alpha-C^{-1}
\end{align*}
for all $\lambda>\lambda_1$. 

This contradicts the definition of $\alpha$.

Thus we have
\begin{align*}
\sup_{\mathbb{R}^n}\Delta^2_eu= \frac{e^\prime A e}{\|e\|^2}.
\end{align*}

\end{proof}

\medskip

To proceed, we choose $b\in \mathbb{R}^n$ such that
\begin{align*}
w(e_k)=w(-e_k),\quad 1\leq k\leq n,
\end{align*}
where 
\begin{align*}
w(x):=u(x)-\frac{1}{2}x^T Ax-b\cdot x.
\end{align*}

By Theorem \ref{existence}, there exists $v\in C^{0,1}(\mathbb{R}^n)$ which is $1$-periodic satisfying $\det(A+D^2v)=f$ in the Alexandrov sense. Choose $v$ such that $v(0)=w(0)$.

Define
\begin{align}\label{def h}
h=w-v.
\end{align}

Then we have $h(0)=0$.
\medskip

We now prove that $h$ is bounded from above. 

\begin{lemm}\label{h bound}
\begin{align*}
\sup_{\mathbb{R}^n}h<\infty.
\end{align*}
\end{lemm}

\begin{proof}
We follow the proof of Lemma 2.9 in \cite{CL2}. On the other hand, since uniform $C^2$ estimates are not available for $f\in L^\infty$, we need to provide new arguments in several places.

Let 
\begin{align*}
{M}_i=\sup_{x\in [-i,i]^n}h(x),\quad i=1,2,\cdots
\end{align*}
Suppose $h$ is not bounded above, then we have
\begin{align*}
\lim_{i\rightarrow \infty}{M}_i=\infty.
\end{align*}
We claim that for some cosntant $C$ independent of $i$, we have
\begin{align}\label{compare m}
{M}_{2^i}\leq 4{M}_{2^{i-1}}+C,\quad \forall i=1,2,\cdots
\end{align}

First of all, since both $w$ and $v$ are locally Lipschitz and $h(0)=0$, we have
\begin{align*}
|h(x)|\leq C,\quad \forall x\in [-1,1]^n.
\end{align*}

Now for $x=(x_1,\cdots,x_n)\in [-m,m]^n$ where $m$ is an integer, let $[x_k]$ be the integer part of $x_k$. Define
\begin{align*}
\epsilon_k=\begin{cases}
1,\quad \textit{if } [x_k] \textit{ is odd},\\
0,\quad \textit{if } [x_k] \textit{ is even}.
\end{cases}
\end{align*}

Then by Proposition \ref{concavity}, we have
\begin{align}\label{Delta h}
\Delta_e^2h=\Delta_e^2w\leq 0,\quad in \quad \mathbb{R}^n,\quad e\in E.
\end{align}

Thus
\begin{align*}
h(x)+h(x-\sum_{k=1}^n ([x_k]+\epsilon_k)e_k)\leq 2h(x-\sum_{k=1}^n \frac{[x_k]+\epsilon_k}{2}e_k).
\end{align*}

Since
\begin{align*}
x-\sum_{k=1}^n ([x_k]+\epsilon_k)e_k\in [-1,1]^n,\quad x-\sum_{k=1}^n \frac{[x_k]+\epsilon_k}{2}e_k\in [-[\frac{m+1}{2}]-1,[\frac{m+1}{2}]+1]^n,
\end{align*}

we have
\begin{align*}
h(x)\leq 2{M}_{[\frac{m+1}{2}]+1}+C.
\end{align*}

It follows that
\begin{align*}
{M}_m\leq 2{M}_{[\frac{m+1}{2}]+1}+C.
\end{align*}

Taking $m=2^i$, we have proved (\ref{compare m}).

Let
\begin{align*}
{H}_i(x)=\frac{h(2^ix)}{{M}_{2^i}},\quad x\in [-1,1]^n.
\end{align*}

By Lemma A.3 in \cite{CL2}, (\ref{Delta h}) and the fact that $h(0)=0$, $h(e_k)=h(-e_k)$, we have
\begin{align}\label{1/2}
{H}_i(\pm\frac{1}{2}e_k)=\frac{h(\pm2^{i-1}e_k)}{{M}_{2^i}}\leq 0,\quad 1\leq k\leq n,\quad i=1,2,\cdots
\end{align}

By (\ref{compare m}), we have
\begin{align}\label{1/8}
\max_{[-\frac{1}{2},\frac{1}{2}]^n}{H}_i=\frac{{M}_{2^{i-1}}}{{M}_{2^i}}\geq \frac{{M}_{2^i}-C}{4{M}_{2^i}}\geq \frac{1}{8}
\end{align}
for large $i$.

By the definition of $H_i$,
\begin{align*}
{H}_i\leq 1\quad on \quad [-1,1]^n,
\end{align*}
and
\begin{align*}
{H}_i(0)=\frac{h(0)}{{M}_{2^i}}=0.
\end{align*}

\medskip

{\it Claim}: Let $0<b^\prime<b\leq 1$, if $l(x)-H_i\geq 0$ in $[-b,b]^n$ for a linear function $l(x)$ , then for some positive constants $\alpha$ and $C$ independent of $i$ and $l(x)$, we have
\begin{align*}
\max_{[-b^\prime,b^\prime]^n}(l-H_i)\leq C\min_{[-b^\prime,b^\prime]^n}(l-H_i),
\end{align*} 
and
\begin{align*}
\|H_i\|_{C^\alpha([-b^\prime,b^\prime]^n)}\leq C.
\end{align*}

\medskip

We now prove the claim.

Recall that $v$ in (\ref{def h}) is the unique solution of $\det(A+D^2v)=f$ in $\mathbb{T}^n$ satisfying $v(0)=w(0)$.

As in the proof of Theorem \ref{existence}, deonte
\begin{align*}
\tilde{f}_\epsilon=f_\epsilon-\strokedint_{\mathbb{T}^n}f_\epsilon+\det A.
\end{align*} 

Let $\tilde{v}_\epsilon$ be the unique function with $(A+D^2\tilde{v}_\epsilon)>0$ satisfying
\begin{align*}
\begin{cases}
\det(A + D^2\tilde{v}_\epsilon)=\tilde{f}_\epsilon,\quad in\quad \mathbb{T}^n,\\
\tilde{v}_\epsilon(0)=v(0)=w(0).
\end{cases}
\end{align*}
Since $|\nabla v_\epsilon|\leq C(A)$ and $\tilde{f}_\epsilon\rightarrow f$  in $C^0(\mathbb{T}^n)$ as $\epsilon\rightarrow 0$, by the uniqueness of solution of $\det(A+D^2v)=f$ on $\mathbb{T}^n$ in the Alexandrov sense, we have $\tilde{v}_\epsilon\rightarrow v$ in $C^0(\mathbb{T}^n)$ as $\epsilon\rightarrow 0$.

For $i$ fixed, denote
\begin{align*}
\Omega_i=\{x\in \mathbb{R}^n|u(x)<C2^{2i-1}\}.
\end{align*}
where $C$ is a fixed constant greater than the largest eigenvalue of $A$. By Proposition \ref{asymp bound}, we have $[-2^i,2^i]^n \subset \Omega_i \subset [-C2^i,C2^i]^n$, where $C$ is another constant depending only on $A$.

Let $\tilde{u}_\epsilon$ be the solution of the following Dirichlet problem
\begin{align*}
\begin{cases}
\det(D^2\tilde{u}_\epsilon(x))=\tilde{f}_\epsilon(x),\quad in \quad \Omega_i,\\
\tilde{u}_\epsilon(x)=M,\quad on\quad \partial\Omega_i.
\end{cases}
\end{align*}
As before, $\tilde{f}_\epsilon\rightarrow f$  in $C^0(\mathbb{T}^n)$ as $\epsilon\rightarrow 0$, and we have $\tilde{u}_\epsilon\rightarrow u$ in $C^0(\bar{\Omega}_i)$ as $\epsilon\rightarrow 0$.

Denote
\begin{align*}
\tilde{h}_\epsilon(x)=\tilde{u}_\epsilon(x)-\frac{1}{2}x^T Ax-bx-\tilde{v}_\epsilon(x).
\end{align*}
It follows that $\tilde{h}_\epsilon\rightarrow h$ in $C^0(\bar{\Omega}_i)$ as $\epsilon\rightarrow 0$.

Recall that
\begin{align*}
F(A + D^2\tilde{v}_\epsilon)\leq F(D^2\tilde{u}_\epsilon)+F_{ij}(D^2\tilde{u}_\epsilon)(A+\partial_{ij}\tilde{v}_\epsilon-\partial_{ij}\tilde{u}_\epsilon),
\end{align*}
i.e.
\begin{align}\label{supersol}
F_{ij}(D^2\tilde{u}_\epsilon)\partial_{ij} \tilde{h}_\epsilon\leq 0.
\end{align}
Similarly,
\begin{align*}
F(D^2\tilde{u}_\epsilon)\leq F(A + D^2\tilde{v}_\epsilon)+F_{ij}(A + D^2\tilde{v}_\epsilon)(\partial_{ij}\tilde{u}_\epsilon-A-\partial_{ij}\tilde{v}_\epsilon),
\end{align*}
i.e.
\begin{align}\label{subsol}
F_{ij}(A+D^2\tilde{v}_\epsilon)\partial_{ij} \tilde{h}_\epsilon\geq 0.
\end{align}

Define
\begin{align*}
\tilde{H}_{\epsilon i}(x)=\frac{\tilde{h}_\epsilon(2^ix)}{M_{2^i}},\quad x\in [-1,1]^n.
\end{align*}

Then $\tilde{H}_{\epsilon i}\rightarrow H_i$ in $C^0([-1,1]^n)$ as $\epsilon\rightarrow 0$.

For any $\delta>0$, we have $l+\delta-\tilde{H}_{\epsilon i}$ is nonnegative in $[-b,b]^n$ for all $\epsilon$ small enough. 

By (\ref{supersol}) and (\ref{subsol}) , we have
\begin{align*}
&F_{ij}(A+D^2\tilde{v}_\epsilon)\partial_{ij}\left( l+\delta-\tilde{H}_{\epsilon i}\right)\geq 0,\quad in \quad\frac{1}{2^i}\Omega_i,\\
&F_{ij}(D^2\tilde{u}_\epsilon)\partial_{ij}\left( l+\delta-\tilde{H}_{\epsilon i}\right)\leq 0,\quad in \quad\frac{1}{2^i}\Omega_i.
\end{align*}
By our choice of $\Omega_i$, we have $[-1,1]^n\subset \frac{1}{2^i}\Omega_i\subset [-C,C]^n$.

By Theorem C, we have
\begin{align*}
\max_{[-b^\prime,b^\prime]^n}(l+\delta-\tilde{H}_{\epsilon i})\leq C(l+\delta-\tilde{H}_{\epsilon i}(0))\leq 2C,
\end{align*}
\begin{align*}
\|l+\delta-\tilde{H}_{\epsilon i}\|_{C^\alpha([-b^\prime,b^\prime]^n)}\leq C,
\end{align*}
where $\alpha, C$ only depends on $n,\lambda, \Lambda$ and $A$, in particular, $\alpha, C$ does not depend on $\epsilon$ and $i$.

The claim is now proved after sending $\epsilon$ to $0$.

\medskip

It follows that there exist some $0<\alpha^\prime<\alpha<1$ and $H$ such that
\begin{align*}
{H}_i\rightarrow {H}\quad in \quad C^{\alpha^\prime} ([-\frac{3}{4},\frac{3}{4}]^n)\textit{ along a subsequence }i\rightarrow\infty.
\end{align*}

By (\ref{1/8}), we have
\begin{align}\label{1/8-2}
\max_{[-\frac{1}{2},\frac{1}{2}]^n}{H}\geq \frac{1}{8}.
\end{align}
By (\ref{1/2}), we have
\begin{align}\label{1/2-2}
{H}(\pm\frac{1}{2}e_k)\leq 0,\quad 1\leq k\leq n.
\end{align}

We also know that
\begin{align*}
{H}(0)=\lim_{i\rightarrow \infty}{H}_i(0)=0.
\end{align*}

By (\ref{Delta h}), 
\begin{align*}
\Delta^2_{2^{-i}e}{H}_i=\frac{\Delta^2_eh}{M_{2_i}}\leq 0,\quad\forall e\in E.
\end{align*}

It follows that ${H}$ is concave. We can then find a linear function $l$ such that $l-{H}\geq 0$ in $[-\frac{3}{4},\frac{3}{4}]^n$ with $l(0)=0$. By the convergence of $H_i$ to $H$, there exist constants $\delta_i\rightarrow 0$ such that $l_i(x)=l(x)+\delta_i$ satisfies $l_i-H_i\geq 0$ in $[-\frac{3}{4},\frac{3}{4}]^n$. Applying the earlier claim to $l_i-H_i$ with $b=\frac{3}{4}$ and $b^\prime=\frac{1}{2}$, and then sending $i$ to $\infty$, we conclude that $\max_{[-\frac{1}{2},\frac{1}{2}]^n} (l-{H})\leq C(l(0)-H(0))=0$. Thus
\begin{align*}
{H}=\sum_{k=1}^n c_kx_k,\quad on \quad [-\frac{1}{2},\frac{1}{2}]^n.
\end{align*}
Now by (\ref{1/2-2}), we conclude that $c_k=0$, i.e. ${H}\equiv 0$. However, this violates (\ref{1/8-2}). The lemma is now proved.
\end{proof}

\medskip

Proof of the Theorem \ref{main theorem}.
\begin{proof}
By Lemma \ref{h bound}, there exists some constant $a$ such that 
\begin{align*}
\inf_{\mathbb{R}^n}(a-h)=0.
\end{align*}
Since $\frac a{ M_{2^i} } -H_i=\frac{a-h(2^ix)}{M_{2^i}}\geq 0$, by the earlier claim, there exists some constant $C$ such that
\begin{align*}
\max_{ [-\frac{1}{2}, \frac{1}{2}]^n }\left( \frac a{ M_{2^i} } -H_i\right)\leq C\min_{  [-\frac{1}{2}, \frac{1}{2}]^n }\left( \frac a{ M_{2^i} } -H_i\right)
\end{align*}
for all large $i$.

Namely,
\begin{align*}
\max_{ [-2^{i-1}, 2^{i-1}]^n }(a-h)\leq C\min_{ [-2^{i-1}, 2^{i-1}]^n }(a-h)
\end{align*}
for all large $i$.

It follows that
\begin{align*}
\sup_{ R^n} (a-h)\le C \inf_{R^n}(a-h)=0.
\end{align*}
Thus $h\equiv a$, i.e. $u\equiv \frac{1}{2}x^TAx+b\cdot x+a+v$.

\end{proof}

\end{document}